\documentclass[12pt]{amsart}

\usepackage[T1]{fontenc}
\usepackage{lmodern}
\usepackage{microtype}

\usepackage[margin=1.15in]{geometry}

\usepackage{amsmath,amssymb,amsthm,mathtools}

\usepackage{enumitem}
\usepackage{hyperref}
\usepackage[nameinlink,capitalize,noabbrev]{cleveref}
\hypersetup{colorlinks=true,linkcolor=blue,citecolor=blue,urlcolor=blue}
\setlist[itemize]{leftmargin=2.1em}
\setlist[enumerate]{leftmargin=2.1em}

\newtheorem{theorem}{Theorem}[section]
\newtheorem{lemma}[theorem]{Lemma}

\theoremstyle{definition}

\theoremstyle{remark}

\newcommand{\RR}{\mathbb{R}}
\newcommand{\PP}{\mathbb{P}}
\newcommand{\EE}{\mathbb{E}}

\newcommand{\abs}[1]{\left|#1\right|}

\newcommand{\bigset}[1]{\Bigl\{#1\Bigr\}}
\newcommand{\floor}[1]{\left\lfloor #1\right\rfloor}
\newcommand{\Sig}{\Sigma}

\title{Distinct permutation dot products}
\author{Cosmin Pohoata}
\date{}

\begin{document}
\maketitle

\begin{abstract}
We show that for any two sets of reals numbers $A=\{a_1,\dots,a_n\}$ and $B=\{b_1,\dots,b_n\}$, the sums of the form $\sum_{i=1}^n a_i\,b_{\pi(i)}$ always take on $\Omega(n^{3})$ distinct values, as we range over all permutations $\pi \in S_n$.

An important ingredient is a ``supportive'' version of Hal\'asz's anticoncentration theorem from Littlewood-Offord theory, which may be of independent interest. 
\end{abstract}

\section{Introduction} \label{sec:intro}

Let $A=\{a_1,\dots,a_n\}\subset\RR$ and $B=\{b_1,\dots,b_n\}\subset\RR$ be sets of $n$ distinct reals.
For $\pi\in S_n$ define the \emph{permutation dot product}
\[
S(\pi):=\sum_{i=1}^n a_i\,b_{\pi(i)},
\qquad
\Sig(A,B):=\bigset{S(\pi):\pi\in S_n}.
\]
Equivalently, writing $a=(a_1,\dots,a_n) \in \mathbb{R}^{n}$ and $b=(b_1,\dots,b_n) \in \mathbb{R}^{n}$, we have
$S(\pi)=\langle a, b_\pi\rangle$ where $b_\pi=(b_{\pi(1)},\dots,b_{\pi(n)}) \in \mathbb{R}^{n}$, so $\Sig(A,B)$ is the
set of values taken by a fixed linear functional on the permutation orbit of $b$. Geometrically, $|\Sig(A,B)|$ can also be thought of as the number of (parallel) hyperplanes with normal vector $(a_1,\ldots,a_n)$ that are needed to cover the vertices of the permutohedron $\operatorname{conv}\left\{b_{\pi}:\ \pi \in S_n\right\}$. 

In this paper, we determine the minimum size of $\abs{\Sig(A,B)}$, as $A,B$ range over all sets of $n$ distinct reals. 

\begin{theorem}\label{thm:main}
Let $A,B\subset\RR$ be sets of $n$ distinct real numbers. 
\[
\abs{\Sig(A,B)}\ \geq cn^{3},
\]
where $c >0$ is an absolute constant.
\end{theorem}

It is not difficult to see that Theorem \ref{thm:main} is optimal, up to the precise value of $c$. Indeed, if $A=B=\{1,2,\dots,n\}$, then for any permutation $\pi \in S_n$, we have that
$$\frac{n(n+1)(n+2)}{6} =\sum_{i=1}^{n} i(n+1-i) \leq S(\pi)=\sum_{i=1}^n i\,\pi(i) \leq \sum_{i=1}^{n}{i^2} = \frac{n(n+1)(2n+1)}{6}.$$
This follows for example from the rearrangement inequality (see \cite{HLP52}). Hence $\abs{\Sigma([n],[n])}$ is on the order of $n^3$. 

Theorem \ref{thm:main} is related in some sense to a new Erd\H{o}s-Moser-type theorem for permutations, which arose recently in two independent works on the permutation anticoncentration model: the work of the author with Hunter and Zhu \cite{HPZ26} and the work by Do, Nguyen, Phan, Tran and Vu \cite{DNPTV25}. This result states the following:

\begin{theorem}[HPZDNPTV]\label{thm:EM-context}
For all sets $A,B\subset\RR$ of size $n$ with distinct elements, and any unformly random permutation $\pi$ on $S_n$
\[
\sup_{x\in\RR}\PP\bigl(S(\pi)=x\bigr) \le\ n^{-5/2+o(1)}.\qquad
\]
\end{theorem}
A trivial observation is the fact that
\begin{equation*}
\abs{\Sig(A,B)}\ \ge\ \frac{1}{\sup_{x\in\RR}\PP\bigl(S(\pi)=x\bigr)},
\end{equation*}
so Theorem \ref{thm:EM-context} immediately implies that $\abs{\Sig(A,B)}\ge n^{5/2-o(1)}$. Nevertheless, Theorem \ref{thm:EM-context} is also sharp (up to the subpolynomial factor) because of the case $A=B=[n]$. Therefore, the point of Theorem~\ref{thm:main} is to improve this first estimate to the correct cubic scale.

On a global level, the proof of Theorem ~\ref{thm:main} will be morally similar to the proof of Theorem \ref{thm:EM-context} from \cite{HPZ26}, however the details will be locally entirely different (and sometimes much simpler). 

\subsection*{Proof overview}

The basic idea is as follows. Fix a base permutation $\pi_0$, and first note that composing $\pi_0$ with a disjoint transposition $(ij)$ changes
$S(\pi)$ by a rectangular area
\[
(a_i-a_j)\bigl(b_{\pi_0(j)}-b_{\pi_0(i)}\bigr) \in (A-A)(B-B).
\]
Choosing many disjoint transpositions produces a (translate of) a collection of subset sums $\Sig(D)$ living inside
$\Sig(A,B)$, where $D \subset (A-A)(B-B)$ is set of ``swap increments''. 

In order to make $\Sig(D)$ as large as possible, we will first develop a
``support version'' of Hal\'asz's inequality from Littlewood-Offord theorem that lower bounds $\abs{\Sig(D)}$ in terms of the additive energy $E_2(D)$ (Section \ref{sec:supportive}). It is perhaps important to highlight that Hal\'asz's theorem by itself can also technically provide a lower bound $\abs{\Sig(D)}$ in terms of $E_{2}(D)$, but that estimate won't suffice for exactly the same reasons Theorem \ref{thm:EM-context} can't directly prove Theorem \ref{thm:main}. 
Finally, we show that one can select $m\asymp n/\sqrt{\log n}$ disjoint swap increments from the pool
$(A-A)(B-B)$ with energy $E_2(D)\ll m^2$. This selection uses a theorem of Roche--Newton and Rudnev \cite[Theorem 1']{RNR}, which in turn relies on a deep incidence theorem for points and lines in $\mathbb{R}^{3}$ due to Guth and Katz from \cite{GuthKatz15} (see also \cite{BloomWalker19} for a simpler proof that only relies on the Szemer\'edi-Trotter theorem). Putting everything together will then yield the slightly weaker bound $\abs{\Sig(D)}\gg m^{5} / E_{2}(D) \gg m^3 \asymp n^{3} / (\log n)^{3/2}$. We discuss this argument first in Sections \ref{sec:subcube} and \ref{sec:lossy}. 

The logarithmic loss in this initial estimate comes from the fact that the one--switch increment
set $(A-A)(B-B)$ can be as small as $n^2/(\log n)^{\delta_0+o(1)}$, where $\delta_0 = 0.86071\ldots$ denotes the Erd\H{o}s-Tenenbaum-Ford constant (e.g.\ for $A=B=[n]$).  To complete the proof of Theorem \ref{thm:main}, we bypass this obstruction by using products of {\it{pairs}} of disjoint transpositions for the ``swap increments''. The resulting increment will be a sum of two rectangle areas and will therefore lie
in the set
\[
(A-A)(B-B)+(A-A)(B-B).
\]
The main point is that this expander turns out to be quadratically large, with no logarithmic loss. The rest of the argument will then use a similar strategy as above: we extract a slightly larger set $D$ of $m \asymp n$ (rather than $n/\sqrt{\log n}$) disjoint swap increments from the pool
$(A-A)(B-B)+(A-A)(B-B)$, with additive energy $E_2(D)\ll m^2$. Our supportive Hal\'asz theorem will then convert this into the lower bound $\abs{\Sig(D)}\gg m^{5} / E_{2}(D) \gg m^3 \asymp n^{3}$, as desired. We will discuss all this in Section \ref{sec:cubic}. 

Last but not least, it is perhaps important to highlight that in order show $|(A-A)(B-B)+(A-A)(B-B)| \gg |A||B|$, we will no longer require the Guth-Katz theorem or the Szemer\'edi-Trotter theorem. The proof will be completely elementary, but in some sense is inspired from Solymosi's well-known ``sector argument'' from \cite{Solymosi09}.

\smallskip
\subsection*{Notation} We use the notation $[N]$ to denote the set $\left\{1,\ldots,N\right\}$. We use $f \ll g$ to denote $f \leq Cg$ for some $C$, where the constant $C$ may depend on subscripts on the $\ll$. We use $f \asymp g$ to denote $f \ll g$ and $f \gg g$, with subscripts treated similarly.

\smallskip
\subsection*{Acknowledgments} 
The author was supported by NSF grant DMS-2246659. The author would also like to thank Zach Hunter, Oliver Roche-Newton, and Daniel Zhu for useful discussions.  

\smallskip
\section{A supportive Hal\'asz theorem} \label{sec:supportive}

For a finite set $D\subset\RR$, formally define its subset-sum set
\[
\Sig(D):=\bigset{\sum_{x\in S}x:\ S\subseteq D}.
\]
For an integer $k\ge 1$, define the $k$-th additive energy
\[
E_k(D):=\#\bigset{(a_1,\dots,a_k,b_1,\dots,b_k)\in D^{2k}:\ a_1+\cdots+a_k=b_1+\cdots+b_k}.
\]
This is a fundamental parameter that appears essentially everywhere throughout additive combinatorics, see for example \cite{TaoVu06}. We are now ready to state the main result of this section. 

\begin{theorem}\label{thm:support-halasz}
Let $D \subset\RR$ be finite with $\abs{D}=m$, and let $k\ge 1$ be fixed. Then
\[
\abs{\Sig(D)}\ \gg_k\ \frac{m^{2k+1}}{E_k(D)}.
\]
\end{theorem}

For instance, if $D$ is Sidon then $E_2(D)\asymp m^2$ and Theorem \ref{thm:support-halasz} for $k=2$ gives $\abs{\Sig(D)}\gg m^3$. This is sharp because there exist Sidon sets $D$ of size $m \asymp n^{1/2}$ in $[n]$: indeed, if $D \subset [n]$ then $\Sigma(D) \subset [|D|n]$ and $|D|n \asymp m^{3}$. This particular estimate for subset sums of Sidon sets is due to Balogh-Lavrov-Shakan-Wagner \cite{BLSW18}.

Theorem \ref{thm:support-halasz} should be compared to the appropriate point-mass version of Hal\'asz theorem, which only implies a polynomially lossy lower bound for $\abs{\Sig(D)}$. See for example \cite[Corollary 6.3]{NguyenVu13} and the references therein for more details. 

\subsection*{Proof of Theorem~\ref{thm:support-halasz}}

Write $m=\abs{D}$. At least $\lceil m/2\rceil$ elements of $D$ have the same sign; replacing $D$ by
$-D$ if needed, we may assume that $D$ contains a subset $P\subseteq D$ of positive elements with
\[
\abs{P}=M\ge m/2.
\]
Since $\Sig(P)\subseteq \Sig(D)$ and $E_k(P)\le E_k(D)$, it suffices to prove
\[
\abs{\Sig(P)}\ \gg_k\ \frac{M^{2k+1}}{E_k(D)}.
\]

Write $P=\{p_1<\cdots<p_M\}\subset\RR_{>0}$. For $t\ge 0$ define
\[
L_t:=p_{M-kt+1}+p_{M-kt+2}+\cdots+p_M,
\qquad L_0:=0,
\]
and define the remaining prefix
\[
P_t:=\{p_1,\dots,p_{M-kt}\}.
\]
Let $k^{\wedge}P_t$ denote the set of sums of $k$ \emph{distinct} elements of $P_t$:
\[
k^{\wedge}P_t:=\bigset{x_1+\cdots+x_k:\ x_1,\dots,x_k\in P_t\ \text{distinct}}.
\]
Define the ``block''
\[
\mathcal{S}_t:=L_t+k^{\wedge}P_t\ \subseteq\ \Sig(P).
\]
The main idea behind the proof is in some sense given by the following simple observation.

\begin{lemma}\label{lem:block-disjoint}
For $t=0,1,\dots,\floor{M/k}-1$, the sets $\mathcal{S}_t$ are pairwise disjoint.
\end{lemma}

\begin{proof}
The maximum element of $k^{\wedge}P_t$ is the sum of the $k$ largest elements of $P_t$, namely
$p_{M-kt}+\cdots+p_{M-kt-k+1}$. Therefore
\[
\max(\mathcal{S}_t)=L_t+p_{M-kt}+\cdots+p_{M-kt-k+1}=L_{t+1}.
\]
Every element of $\mathcal{S}_{t+1}$ equals $L_{t+1}$ plus a sum of $k$ positive elements, hence is
strictly larger than $L_{t+1}=\max(\mathcal{S}_t)$. This shows
$\max(\mathcal{S}_t)<\min(\mathcal{S}_{t+1})$ and proves the claim.
\end{proof}

By Lemma~\ref{lem:block-disjoint},
\begin{equation}\label{eq:block-sum-general}
\abs{\Sig(D)}\ \ge\ \abs{\Sig(P)}\ \ge\ \sum_{t=0}^{\floor{M/k}-1}\abs{k^{\wedge}P_t}.
\end{equation}

The next lemma will allow us to give a lower bound for each $\abs{k^{\wedge}P_t}$ in terms of $E_{k}(D)$.

\begin{lemma}\label{lem:k-sum-from-energy}
Let $C\subset\RR$ be finite with $\abs{C}=s\ge k$. Then
\[
\abs{k^{\wedge}C}\ \ge\ \frac{(s)_k^2}{E_k(C)},
\]
where $(s)_k=s(s-1)\cdots(s-k+1)$ is the falling factorial function.
\end{lemma}

\begin{proof}
Let $r_C(x)$ be the number of ordered $k$-tuples of \emph{distinct} elements
$(c_1,\dots,c_k)\in C^k$ with $c_1+\cdots+c_k=x$. Then
\[
\sum_x r_C(x)=(s)_k,
\qquad
\sum_x r_C(x)^2 \le E_k(C).
\]
By Cauchy--Schwarz,
\[
(s)_k^2=\Bigl(\sum_x r_C(x)\Bigr)^2
\le \abs{k^{\wedge}C}\cdot \sum_x r_C(x)^2
\le \abs{k^{\wedge}C}\cdot E_k(C),
\]
which gives the claim.
\end{proof}

Apply Lemma~\ref{lem:k-sum-from-energy} with $C=P_t$ (so $s=M-kt$) and substitute into
\eqref{eq:block-sum-general}:
\[
\abs{\Sig(D)}\ \ge\ \sum_{t=0}^{\floor{M/k}-1}\frac{(M-kt)_k^2}{E_{k}(P_t)} \geq\ \frac{1}{E_k(D)}\sum_{t=0}^{\floor{M/k}-1}(M-kt)_k^2.
\]
In the second inequality we used that $E_k(P) \leq E_k(D)$ for al $t = 0,\ldots,\floor{M/k}-1$. Let us in fact restrict to $0\le t\le \floor{M/(2k)}-1$, for which $M-kt\ge M/2$. Then $(M-kt)_k\gg_k M^k$, and there
are $\gg_k M$ such $t$. Hence the sum is $\gg_k M^{2k+1}$, proving
\[
\abs{\Sig(D)}\gg_k \frac{M^{2k+1}}{E_k(D)}\gg_k \frac{m^{2k+1}}{E_k(D)}.
\]
This completes the proof of Theorem~\ref{thm:support-halasz}.

\section{Disjoint transpositions embed subset sums into $\Sig(A,B)$} \label{sec:subcube}

We now return to the permutation dot-product model. For the reader's convenience, we isolate the following basic (yet crucial) algebraic identity as a Lemma.

\begin{lemma}\label{lem:swap-increment}
Fix $\pi\in S_n$ and positions $i\neq j$. Let $\pi'=\pi\circ (ij)$ be obtained by swapping the images
of $i$ and $j$. Then
\[
S(\pi')-S(\pi)=(a_i-a_j)\bigl(b_{\pi(j)}-b_{\pi(i)}\bigr).
\]
\end{lemma}

\begin{proof}
Only the $i$- and $j$-terms change:
\begin{align*}
S(\pi')-S(\pi)
&=\bigl(a_i b_{\pi(j)}+a_j b_{\pi(i)}\bigr)-\bigl(a_i b_{\pi(i)}+a_j b_{\pi(j)}\bigr)\\
&=(a_i-a_j)\bigl(b_{\pi(j)}-b_{\pi(i)}\bigr).
\end{align*}
\end{proof}

\begin{lemma}\label{lem:disjoint-transpositions}
Fix a base permutation $\pi_0\in S_n$. Let $\tau_t=(i_t\,j_t)$, $t=1,\dots,m$, be pairwise disjoint
transpositions (so all indices $i_1,j_1,\dots,i_m,j_m$ are distinct). Define
\[
\delta_t:=(a_{i_t}-a_{j_t})\bigl(b_{\pi_0(j_t)}-b_{\pi_0(i_t)}\bigr),
\qquad
D:=\{\delta_1,\dots,\delta_m\}.
\]
For each $I\subseteq[m]$, let
\[
\pi_I:=\pi_0\circ\prod_{t\in I}\tau_t.
\]
Then
\[
S(\pi_I)=S(\pi_0)+\sum_{t\in I}\delta_t.
\]
In particular, $\Sig(A,B)$ contains a translate of $\Sig(D)$, which implies $\abs{\Sig(A,B)}\ \ge\ \abs{\Sig(D)}$.
\end{lemma}

\begin{proof}
First, note that since the transpositions $\tau_t=(i_t\,j_t)$ are pairwise disjoint, they act on disjoint sets of indices and therefore commute with each other. In particular, the product $\prod_{t\in I}\tau_t$ is independent of the order in which it is written, and $\pi_I$ is well defined.

Now, fix $t\in[m]$ and consider the permutation $\pi_0\circ\tau_t$. By Lemma \ref{lem:swap-increment}, observe that
\begin{align*}
S(\pi_0\circ\tau_t)-S(\pi_0)
&=(a_{i_t}-a_{j_t})\bigl(b_{\pi_0(j_t)}-b_{\pi_0(i_t)}\bigr)
=\delta_t,
\end{align*}
Hence, for every subset $I=\{t_1,\dots,t_k\} \subset [m]$, then
\[
S(\pi_I)-S(\pi_0)
=\sum_{\ell=1}^k \bigl(S(\pi_0\circ\tau_{t_1}\circ\cdots\circ\tau_{t_\ell})
      -S(\pi_0\circ\tau_{t_1}\circ\cdots\circ\tau_{t_{\ell-1}})\bigr)=\sum_{t \in I} \delta_{t}.
\]
In particular, as $I$ ranges over all subsets of $[m]$, the values
\[
S(\pi_I)=S(\pi_0)+\sum_{t\in I}\delta_t
\]
form exactly the translate $S(\pi_0)+\Sig(D)$. Therefore $\Sig(A,B)$ contains a translate of
$\Sig(D)$, and so
\[
\abs{\Sig(A,B)}\ \ge\ \abs{\Sig(D)}.
\]
\end{proof}

\subsection*{Geometric viewpoint}
Let $b=(b_1,\dots,b_n)$ and consider the \emph{permutohedron}
\[
\mathrm{Perm}(b):=\mathrm{conv}\{\,b_\pi:\ \pi\in S_n\,\}\subset\RR^n,
\qquad
b_\pi:=(b_{\pi(1)},\dots,b_{\pi(n)}).
\]
The preceding lemma can be viewed geometrically as follows.  Fix $\pi_0\in S_n$ and pairwise disjoint
transpositions $\tau_t=(i_t\,j_t)$, $t\in[m]$.  Then the $2^m$ vertices
\[
\{\,b_{\pi_I}:\ I\subseteq[m]\,\}\subset \mathrm{Perm}(b)
\]
form the vertex set of an embedded affine $m$--dimensional hypercube (indeed, a product of segments):
toggling $\tau_t$ swaps the $i_t$-th and $j_t$-th coordinates and leaves all other coordinates fixed,
and the disjointness of the pairs $(i_t,j_t)$ ensures that these ``moves'' act on disjoint coordinate
supports.  Concretely, one has
\[
b_{\pi_I}
=
b_{\pi_0}+\sum_{t\in I}\bigl(b_{\pi_0(j_t)}-b_{\pi_0(i_t)}\bigr)\,(e_{i_t}-e_{j_t}),
\]
so the convex hull of $\{b_{\pi_I}:I\subseteq[m]\}$ is a translate of a box generated by the
independent directions $(e_{i_t}-e_{j_t})$.

The linear functional $x\mapsto \langle a,x\rangle$ maps this affine cube to the translate
$S(\pi_0)+\Sigma(D)$, and the scalars $\delta_t$ are precisely the one-dimensional edge increments
of this projection.  In this sense, Lemma~\ref{lem:disjoint-transpositions} is a manifestation of
the fact that the permutohedron contains hypercubes, and our sumset lower bounds arise from
projecting such cubes onto a line.

\section{A lossy Theorem \ref{thm:main} via rectangular area sampling} \label{sec:lossy}

For points $p=(p_1,p_2),q=(q_1,q_2)\in\RR^2$ define the rectangular area
\[
R(p,q):=(q_1-p_1)(q_2-p_2).
\]
For a finite point set $P\subset\RR^2$ define
\[
R(P):=\{R(p,q):p,q\in P\}.
\]

\begin{theorem}[Roche--Newton--Rudnev]\label{thm:RNR}
Let $P\subset\RR^2$ be a set of $N$ points not contained in a single horizontal or vertical line.
Then
\[
\abs{R(P)}\ \gg\ \frac{N}{\log N}.
\]
\end{theorem}

Theorem~\ref{thm:RNR} is a specialization of \cite[Theorem~1$'$]{RNR}, proved via a Cauchy--Schwarz
argument and a bound on ``rectangular quadruples'' (see \cite[Proposition~4]{RNR}).

We will be applying Theorem~\ref{thm:RNR} to the grid $P=A\times B$ (so $N=n^2$). Then
\[
R(A\times B)=\{(a'-a)(b'-b):a,a'\in A,\ b,b'\in B\}=(A-A)(B-B),
\]
and hence
\begin{equation}\label{eq:AA-BB-lower}
\abs{(A-A)(B-B)}\ \gg\ \frac{n^2}{\log n}.
\end{equation}

\medskip
\subsection*{Game plan} We will construct $m$ disjoint swap increments $\delta_1,\dots,\delta_m$ from $(A-A)(B-B)$ so that
their additive energy satisfies $E_2(D)\ll m^2$. Using Theorem \ref{thm:RNR}, will be able to do this for $m \asymp n/\sqrt{\log n}$. Once this is achieved, the case $k=2$ of
Theorem~\ref{thm:support-halasz} will yield $\abs{\Sig(D)}\gg m^3 \asymp n^{3}/(\log n)^{3/2}$, and then
Lemma~\ref{lem:disjoint-transpositions} gives $\abs{\Sig(A,B)}\ge \abs{\Sig(D)}$. 

\smallskip

For now, let's discuss how to find the set $D = \left\{\delta_1,\ldots,\delta_m\right\}$ with the desired properties above. We fix
\begin{equation}\label{eq:def-M}
R := c_0\,\frac{n^2}{\log n},
\end{equation}
with $c_0>0$ small enough that any subgrid $A_t\times B_t$ with $\abs{A_t},\abs{B_t}\ge n/2$ satisfies
\(
\abs{(A_t-A_t)(B_t-B_t)}\ge R.
\)

\subsection*{Algorithm}

Initialize $A_0:=A$, $B_0:=B$. For $t=1,2,\dots,m$, we run the following procedure:
\begin{enumerate}
\item Let $P_{t-1}:=A_{t-1}\times B_{t-1}$ and choose any subset $U_{t-1}\subseteq R(P_{t-1})\setminus\{0\}$ with $\abs{U_{t-1}}=R$.
\item Choose $\delta_t$ uniformly at random from $U_{t-1}\setminus\{\delta_1,\dots,\delta_{t-1}\}$.
\item Choose a pair of points
\[
p=(a_{j_t},b_{p_t}),\quad q=(a_{i_t},b_{q_t})\in P_{t-1}
\quad\text{such that}\quad
\delta_t=R(p,q)=(a_{i_t}-a_{j_t})(b_{q_t}-b_{p_t}).
\]
\item Remove the used elements from $A_{t-1}$ and $A_{t}$:
\[
A_t:=A_{t-1}\setminus\{a_{i_t},a_{j_t}\},\qquad
B_t:=B_{t-1}\setminus\{b_{p_t},b_{q_t}\}.
\]
\item Repeat.
\end{enumerate}

After $m$ steps, the index pairs $(i_t,j_t)$ are disjoint and $(p_t,q_t)$ are disjoint, and
$D:=\{\delta_1,\dots,\delta_m\}$ is a set of size $m$. Also, since $m = o(n)$ this process will certainly not stop earlier than it should. Also, we always maintain the property that
$\abs{A_{t}},\abs{B_{t}}\ge n/2$, hence, by the definition of $R$, we will always be able to find $U_{t} \subset R(P_{t-1}) \setminus \left\{0\right\}$ with $|U_{t-1}| = R$. 

\begin{lemma}\label{lem:expected-energy}
For $D$ constructed above,
\[
\EE\,E_2(D)\ \ll\ m^2+\frac{m^4}{R}.
\]
\end{lemma}
In particular, if $m=c\sqrt{R}$ with $c$ small enough, then $m^4/R\asymp c^2m^2$,
so $\EE E_2(D)\ll m^2$, and Markov's inequality gives an outcome with $E_2(D)\ll m^2$. By \eqref{eq:def-M}, this means that $m$ is on the order of $n/\sqrt{\log n}$.

\begin{proof}
Write $E_2(D)=T_{\mathrm{triv}}+T_{\mathrm{nontriv}}$, where $T_{\mathrm{triv}}$ counts solutions of
$x_1+x_2=x_3+x_4$ with $(x_1,x_2,x_3,x_4)=(x,y,x,y)$ or $(x,y,y,x)$. Then $T_{\mathrm{triv}}\le 2m^2$.

Now fix an ordered index quadruple $(\alpha,\beta,\gamma,\delta)\in[m]^4$ which is not of these two
trivial patterns. Let $t_*:=\max\{\alpha,\beta,\gamma,\delta\}$. Conditioning on the history up to
time $t_*-1$, all $\delta_t$ for $t<t_*$ are fixed, while $\delta_{t_*}$ is uniform on a set of size
$\ge R/2$ (since $t_*\le m\ll R$ and we exclude previously chosen values).
The equation $\delta_\alpha+\delta_\beta=\delta_\gamma+\delta_\delta$ then forces $\delta_{t_*}$ to
equal at most one prescribed value, so
\[
\PP(\delta_\alpha+\delta_\beta=\delta_\gamma+\delta_\delta\mid\text{history up to }t_*-1)\ \ll\ \frac{1}{R}.
\]
Summing over $O(m^4)$ nontrivial index quadruples yields $\EE T_{\mathrm{nontriv}}\ll m^4/R$.
Hence $\EE E_2(D)\ll m^2+m^4/R$. 
\end{proof}

\subsection*{A lossy version of Theorem \ref{thm:main}} \label{sec:lossy}

Let $D=\{\delta_1,\dots,\delta_m\}$ be a deterministic outcome of the above construction with
$E_2(D)\ll m^2$, guaranteed by Lemma~\ref{lem:expected-energy}. Recall that $m \asymp m / \sqrt{\log n}$. 

We now realize these increments as disjoint-transposition increments for a suitable base permutation $\pi_0$. Define $\pi_0\in S_n$ by
\[
\pi_0(i_t)=p_t,\qquad \pi_0(j_t)=q_t\qquad (t=1,\dots,m),
\]
and extend arbitrarily on the remaining indices. Then
\[
(a_{i_t}-a_{j_t})\bigl(b_{\pi_0(j_t)}-b_{\pi_0(i_t)}\bigr)
=(a_{i_t}-a_{j_t})(b_{q_t}-b_{p_t})=\delta_t,
\]
so Lemma~\ref{lem:disjoint-transpositions} yields
\[
\abs{\Sig(A,B)}\ \ge\ \abs{\Sig(D)}.
\]
Apply Theorem~\ref{thm:support-halasz} with $k=2$:
\[
\abs{\Sig(D)}\ \gg\ \frac{m^5}{E_2(D)}\ \gg\ m^3 \asymp \frac{n^{3}}{(\log n)^{3/2}}.
\]
Putting things together, we conclude that $|\Sigma(A,B)| \gg n^{3}/(\log n)^{3/2}$. \qed

\section{Proof of Theorem \ref{thm:main}: Removing the logarithmic loss} \label{sec:cubic}

As already alluded to in Section \ref{sec:intro}, the logarithmic loss in the argument from Section \ref{sec:lossy} comes from the fact that the one-switch increment
set $(A-A)(B-B)$ can be as small as $n^2/(\log n)^{\delta_0+o(1)}$, where $\delta_0 = 0.86071\ldots$ denotes the Erd\H{o}s-Tenenbaum-Ford constant. A simple way to
bypass this obstruction is to use, as a single ``cube direction'', the product of two disjoint
transpositions.  The resulting increment is a sum of two rectangle areas, and therefore lies
in the two--area set
\[
(A-A)(B-B)+(A-A)(B-B).
\]
The main point is that this expander turns out to be quadratically large, with no logarithmic loss. We record this important first observation as a lemma.

\begin{lemma}\label{lem:two-area-warmup}
Let $A,B\subset\RR$ be finite sets with $|A|,|B| \geq 2$. Then
\[
\bigl|(A-A)(B-B)+(A-A)(B-B)\bigr|\ \ge\ \frac{|A||B|}{2}.
\]
\end{lemma}

We note that Lemma \ref{lem:two-area-warmup} is optimal, up to the leading constant. Indeed, it is not difficult to see that for $A=B=[n]$, we have that
$$(A-A)(B-B)+(A-A)(B-B) \subset [-2(n-1)^2,\,2(n-1)^2].$$

\begin{proof} 
Write
$$A=\{a_1<a_2<\cdots<a_n\},\qquad B=\{b_1<b_2<\cdots<b_m\},$$
and let $\delta := \min_{i}(a_{i+1}-a_i)$ and $M := b_m-b_1$. In words, $\delta$ represents the smallest gap in $A$ and $M$ is the the diameter of the set $B$. In particular, $\delta\in A-A$ and $M\in B-B$. Furthermore, let
\begin{align*}
P&:=\{b_i-b_1:\ 1\le i\le m\}\subseteq B-B,\\
P'&:=\{a_{2j-1}-a_1:\ 1\le j\le \lceil n/2\rceil\}\subseteq A-A.
\end{align*}
We have $|P|=m$ and $|P'|=\lceil n/2\rceil$. Moreover, $P \subset [0,M]$ and any two distinct elements of $P'$ differ by at least $2\delta$. We now consider the following special subset of $(A-A)(B-B)+(A-A)(B-B)$:
\[
T:=\delta P + M P'=\{\delta x+My:\ x\in P,\ y\in P'\}.
\]
Since $\delta\in A-A$ and $P\subseteq B-B$, we have $\delta P\subseteq (A-A)(B-B)$, and similarly
$MP'\subseteq (A-A)(B-B)$. Hence we indeed have that $T\subseteq (A-A)(B-B)+(A-A)(B-B)$.

The key idea is that the map $P \times P' \to T$ defined by $(x,y)\mapsto \delta x+My$ is injective. Indeed, if
$\delta x_1+My_1=\delta x_2+My_2$, then $M(y_1-y_2)=\delta(x_2-x_1)$.  If $y_1\neq y_2$ then
$|y_1-y_2|\ge 2\delta$, so $|M(y_1-y_2)|\ge 2\delta M$, whereas $x_1,x_2\in P\subseteq[0,M]$ implies
$|\delta(x_2-x_1)|\le \delta M$, a contradiction.  Thus $y_1=y_2$ and then $x_1=x_2$.

It follows that
$$(A-A)(B-B)+(A-A)(B-B) \geq |T| \geq |P|\cdot |P'| = m\cdot \Bigl\lceil\frac{n}{2}\Bigr\rceil \geq \frac{mn}{2}.$$
\end{proof}

\medskip
With Lemma \ref{lem:two-area-warmup} in hand, the plan from now on will simply be to execute the exact same proof strategy from Section \ref{sec:lossy} but with $(A-A)(B-B)+(A-A)(B-B)$ instead of $(A-A)(B-B)$. 

The first technicality that we have to address is the fact that in the paired--switch argument from before we needed slightly more than the size of $|(A-A)(B-B)|$ as input. Similarly, here we will also need that each increment $\delta_t$ to come from a single commuting generator $\sigma_t=(i_t\,j_t)(k_t\,\ell_t)$, i.e.\ it must admit a representation
\[
\delta_t=(a_i-a_j)(b_p-b_q)+(a_k-a_\ell)(b_r-b_s)
\]
with \emph{four distinct} $A$--indices $\{i,j,k,\ell\}$ and \emph{four distinct} $B$--indices
$\{p,q,r,s\}$ (so that the two underlying transpositions are disjoint). The next lemma provides such a technical refinement of Lemma \ref{lem:two-area-warmup}. 

\begin{lemma}\label{lem:two-area-pool}
Let $A,B\subset\RR$ be finite sets with $|A|=n\ge 4$ and $|B|=m\ge 4$. Then there exists a set
\[
U(A,B)\ \subseteq\ (A-A)(B-B)+(A-A)(B-B)
\]
with
\[
|U(A,B)|\ \ge\ (m-3)\Bigl(\Bigl\lceil\frac n2\Bigr\rceil-2\Bigr)\ \ge\ \frac{nm}{8},
\]
such that every $u\in U(A,B)$ admits a representation
\[
u=(a_i-a_j)(b_p-b_q)\;+\;(a_k-a_\ell)(b_r-b_s)
\]
in which $i,j,k,\ell$ are \emph{distinct} indices and $p,q,r,s$ are \emph{distinct} indices.
\end{lemma}

The proof is almost identical (with just some additional book-keeping), but we include the full proof for the reader's convenience. 

\begin{proof}
Like in the proof of Lemma \ref{lem:two-area-warmup}, write $A=\{a_1<\cdots<a_n\}$ and $B=\{b_1<\cdots<b_m\}$, and define 
\[
\delta:=\min_{1\le t\le n-1}(a_{t+1}-a_t)>0,
\qquad
M:=b_m-b_1>0.
\]
Set
\[
P:=\{\,b_i-b_2:\ 3\le i\le m-1\,\}\subseteq B-B,
\]
so $|P|=m-3$, and each $x\in P$ is represented by a difference $b_i-b_2$ using indices in
$\{2,3,\dots,m-1\}$, hence disjoint from $\{1,m\}$.

Choose an index $s \in [n-1]$ with $a_{s+1}-a_s=\delta$, and select an anchor $a_\star\in A\setminus\{a_s,a_{s+1}\}$. For concreteness, take $a_\star=a_1$ if
$a_1\notin\{a_s,a_{s+1}\}$ and otherwise take $a_\star=a_n$.  Let $I$ be the set of odd indices
$2j-1$ with $1\le j\le\lceil n/2\rceil$, and delete from $I$ any indices corresponding to the elements $\{a_s,a_{s+1}\}$ and also delete the anchor index itself.
Define
\[
P':=\{\,a_t-a_\star:\ t\in I\,\}\subseteq A-A.
\]
Then $|P'|\ge \lceil n/2\rceil-2$.  Moreover, any two distinct elements of $P'$ differ by at least
$2\delta$. Like in the proof of Lemma \ref{lem:two-area-warmup}, consider
\[
U:=\delta P+MP'=\{\delta x+My:\ x\in P,\ y\in P'\} \subset (A-A)(B-B)+(A-A)(B-B).
\]
The map $(x,y)\mapsto \delta x+My$ is still an injective map on $P\times P'$, therefore 
$$|U| \geq |P||P'|\ge (m-3)(\lceil n/2\rceil-2).$$

Most importantly, note that every $u=\delta x+My\in U$ admits a representation using disjoint indices: write
$x=b_i-b_2$ with $3\le i\le m-1$ and write $y=a_t-a_\star$ with $a_t\notin\{a_s,a_{s+1}\}$ and
$a_\star\notin\{a_s,a_{s+1}\}$.  Then
\[
u=(a_{s+1}-a_s)(b_i-b_2)\;+\;(a_t-a_\star)(b_m-b_1),
\]
and the $A$--indices $\{s,s+1,t,\star\}$ are distinct by construction, while the $B$--indices
$\{i,2,m,1\}$ are distinct.  Taking $U(A,B):=U$ completes the proof.
\end{proof}

\medskip

\subsection*{Paired switches give cubes and subset sums}
The next lemma is the analogue of Lemma~\ref{lem:disjoint-transpositions} when each generator is a
product of two disjoint transpositions. The point is that the paired transpositions still give a translate of $\Sigma(D)$ inside $\Sigma(A,B)$. 

\begin{lemma}\label{lem:paired-switches}
Fix a base permutation $\pi_0\in S_n$.  Let
\[
\sigma_t:=(i_t\,j_t)(k_t\,\ell_t),\qquad t=1,\dots,m,
\]
be permutations where all indices $i_t,j_t,k_t,\ell_t$ are distinct across all $t$ (so the supports
are pairwise disjoint and the $\sigma_t$ commute).  Define
\[
\delta_t:=(a_{i_t}-a_{j_t})\bigl(b_{\pi_0(j_t)}-b_{\pi_0(i_t)}\bigr)
      +(a_{k_t}-a_{\ell_t})\bigl(b_{\pi_0(\ell_t)}-b_{\pi_0(k_t)}\bigr),
\qquad
D:=\{\delta_1,\dots,\delta_m\}.
\]
For each $I\subseteq[m]$, let $\pi_I=\pi_0\circ\prod_{t\in I}\sigma_t$. Then
\[
S(\pi_I)=S(\pi_0)+\sum_{t\in I}\delta_t.
\]
In particular, $\Sigma(A,B)$ contains a translate of $\Sigma(D)$ and hence
$\abs{\Sigma(A,B)}\ge \abs{\Sigma(D)}$.
\end{lemma}

\begin{proof}
Since the supports are disjoint, the $\sigma_t$ commute and the product $\prod_{t\in I}\sigma_t$ is
well defined.  Moreover, toggling $\sigma_t$ only affects the four positions
$\{i_t,j_t,k_t,\ell_t\}$, and it is the composition of the two disjoint swaps $(i_t\,j_t)$ and
$(k_t\,\ell_t)$.  Applying Lemma~\ref{lem:swap-increment} to each swap and adding the two increments
gives
\[
S(\pi_0\circ\sigma_t)-S(\pi_0)=\delta_t.
\]
Because the supports are disjoint, these increments add when multiple $\sigma_t$ are applied, giving
the stated formula for $S(\pi_I)$.  The translate containment follows immediately.
\end{proof}

We are finally ready to complete the proof of Theorem \ref{thm:main}. 

\medskip
{\it{Proof of Theorem \ref{thm:main}.}} Roughly speaking, we want to finish the argument by combining Lemmas~\ref{lem:two-area-pool} and \ref{lem:paired-switches} with the supportive Hal\'asz Theorem \ref{thm:support-halasz}.

Set $m:=\lfloor n/32\rfloor$.  We iteratively construct disjoint paired switches and increments
$\delta_1,\dots,\delta_m$. Initialize $A_0:=A$, $B_0:=B$.  For $t=1,\dots,m$ do the following.
Since $4m\le n/8$, we have $|A_{t-1}|,|B_{t-1}|\ge n-4(t-1)\ge 7n/8\ge 4$.
Apply Lemma~\ref{lem:two-area-pool} to $(A_{t-1},B_{t-1})$ to obtain a set
\[
U_{t-1}:=U(A_{t-1},B_{t-1})\subseteq (A_{t-1}-A_{t-1})(B_{t-1}-B_{t-1})+(A_{t-1}-A_{t-1})(B_{t-1}-B_{t-1})
\]
with
\[
|U_{t-1}|\ \ge\ \frac{|A_{t-1}||B_{t-1}|}{8}\ \ge\ \frac{n^2}{16}.
\]
Choose $\delta_t$ uniformly at random from $U_{t-1}\setminus\{\delta_1,\dots,\delta_{t-1}\}$, and
fix one of its disjoint representations
\[
\delta_t=(a_{i_t}-a_{j_t})(b_{p_t}-b_{q_t})+(a_{k_t}-a_{\ell_t})(b_{r_t}-b_{s_t})
\]
with $i_t,j_t,k_t,\ell_t$ distinct and $p_t,q_t,r_t,s_t$ distinct (guaranteed by the lemma).
Delete the used elements:
\[
A_t:=A_{t-1}\setminus\{a_{i_t},a_{j_t},a_{k_t},a_{\ell_t}\},\qquad
B_t:=B_{t-1}\setminus\{b_{p_t},b_{q_t},b_{r_t},b_{s_t}\}.
\]
After $m$ steps, the $4m$ row indices and $4m$ column indices used are pairwise distinct. Finally, let $D:=\{\delta_1,\dots,\delta_m\} \subset (A-A)(B-B)+(A-A)(B-B)$.

The proof of Lemma~\ref{lem:expected-energy} applies verbatim
(with the pool size $R$ there replaced by $|U_{t-1}|\ge n^2/16$ here) and gives
\[
\EE\,E_2(D)\ \ll\ m^2+\frac{m^4}{n^2}.
\]
With $m=\lfloor n/32\rfloor$ we have $m^4/n^2\ll m^2$, hence $\EE E_2(D)\ll m^2$, so there exists an
outcome with $E_2(D)\ll m^2$.

For such an outcome, define a base permutation $\pi_0\in S_n$ by prescribing
\[
\pi_0(i_t)=q_t,\ \ \pi_0(j_t)=p_t,\ \ \pi_0(k_t)=s_t,\ \ \pi_0(\ell_t)=r_t
\qquad (t=1,\dots,m),
\]
and extending arbitrarily on the remaining indices.  Then
\[
(a_{i_t}-a_{j_t})\bigl(b_{\pi_0(j_t)}-b_{\pi_0(i_t)}\bigr)=(a_{i_t}-a_{j_t})(b_{p_t}-b_{q_t}),
\]
and similarly for the $(k_t,\ell_t)$--pair, so Lemma~\ref{lem:paired-switches} yields
$\abs{\Sigma(A,B)}\ge \abs{\Sigma(D)}$.

Finally, we apply \Cref{thm:support-halasz} to $D$ like before:
\[
\abs{\Sigma(D)}\ \gg\ \frac{m^5}{E_2(D)}\ \gg\ m^3\ \gg\ n^3,
\]
which completes the proof.
    \qed

\section{Concluding remarks} \label{sec:conclusion}
The permutation dot-product problem admits the following natural matrix generalization.  Given
$M=(m_{ij})\in\RR^{n\times n}$, define
\[
S_M(\pi):=\sum_{i=1}^n m_{i,\pi(i)},\qquad
\Sigma(M):=\{S_M(\pi):\pi\in S_n\}.
\]
Equivalently, writing $P_\pi$ for the permutation matrix of $\pi$ and $\langle\cdot,\cdot\rangle_F$
for the Frobenius inner product, we have
\[
S_M(\pi)=\langle M,P_\pi\rangle_F,
\]
so $\Sigma(M)$ is the set of values taken by the linear functional $X\mapsto\langle M,X\rangle_F$
on the vertex set of the Birkhoff polytope $\mathcal{B}_n=\mathrm{conv}\{P_\pi:\pi\in S_n\}$.

Like the permutohedron, the Birkhoff polytope $\mathcal{B}_n$ also contains many embedded hypercubes: if one fixes a base vertex $P_{\pi_0}$ and a collection of
pairwise disjoint transpositions, then toggling any subset of these transpositions produces
$2^m$ vertices forming the vertex set of an $m$--dimensional cube-face. Projecting such a cube under $X\mapsto\langle M,X\rangle_F$ yields a translate of a subset-sum set,
and then perhaps our supportive Hal\'asz Theorem \ref{thm:support-halasz} could used again to lower bound $\Sigma(M)$ via the
increment set along the cube directions.

At the same time, unlike the rank-one setting $m_{ij}=a_i b_j$, no nontrivial lower bound on
$\abs{\Sigma(M)}$ is possible without additional hypotheses on $M$.
For example, fix two vectors $u,v\in\RR^n$ and let $m_{ij}=u_i+v_j$ for every $1 \leq i,j \leq n$. Then for every $\pi\in S_n$,
\[
S_M(\pi)=\sum_{i=1}^n (u_i+v_{\pi(i)})=\sum_{i=1}^n u_i+\sum_{j=1}^n v_j
\]
is independent of $\pi$, and hence $\abs{\Sigma(M)}=1$.
It would thus be interesting to identify the natural geometric conditions on $M$ that force $\abs{\Sigma(M)}$ to exhibit growth, and in particular to generalize Theorem \ref{thm:main} in the ``correct'' way to this setting.

\end{document}